%% file: proc.tex
\documentclass[11pt,reqno]{amsart} 

\usepackage{amssymb, amsfonts, amsmath, latexsym}

\usepackage[height=190mm,width=130mm]{geometry} 
\usepackage{graphicx}
\usepackage{color}	
\usepackage{tikz}

\theoremstyle{plain}
\newtheorem{theorem}{Theorem}

\newtheorem{corollary}{Corollary}
\newtheorem{proposition}{Proposition}

\theoremstyle{definition}
\newtheorem{definition}{Definition}

\theoremstyle{remark}
\newtheorem{remark}{Remark}

\numberwithin{equation}{section} 

\begin{document}
\newcommand{\pic}{$\spadesuit$}
\newcommand{\N}{\mathbb{N}}
\newcommand{\Z}{\mathbb{Z}}
\newcommand{\C}{\mathbb{C}}
\newcommand{\R}{\mathbb{R}}
\newcommand{\F}{\mathbb{F}}
\newcommand{\defineL}{\,\mathrel{\mathop:}=}
\newcommand{\defineR}{ =\mathrel{\mathop:}\,}
\newcommand{\be}{\begin{equation}}
\newcommand{\ee}{\end{equation}}
\newcommand{\bi}{\begin{itemize}}
\newcommand{\ei}{\end{itemize}}
\newcommand{\with}{\qquad\text{with}\qquad}
\newcommand{\mand}{\qquad\text{and}\qquad}
\newcommand{\sep}{\; \text{,}\qquad}
\newcommand{\ssep}{\: \text{,}\quad}
\newcommand{\pd}{\partial}
\newcommand{\pdo}{\overline{\partial}}
\newcommand{\ov}[1]{\overline{#1}}
\newcommand{\inv}[1]{\frac{1}{#1}}
\newcommand{\tinv}[1]{\tfrac{1}{#1}}
\newcommand{\abs}[1]{|#1|}
\newcommand{\B}[1]{\mathbf{#1}}
\newcommand{\op}[1]{\text{#1}}
\newcommand{\bracket}[1]{\left( #1 \right)}
\newcommand{\bracketi}[1]{\bigl( #1 \bigr)}
\newcommand{\bracketii}[1]{\Bigl( #1 \Bigr)}
\newcommand{\bracketiii}[1]{\biggl( #1 \biggr)}
\newcommand{\bracketiv}[1]{\Biggl( #1 \Biggr)}
\newcommand{\bpm}{\begin{pmatrix}}
\newcommand{\epm}{\end{pmatrix}}
\newcommand{\bnm}{\begin{matrix}}
\newcommand{\enm}{\end{matrix}}
\newcommand{\w}{\wedge}
\newcommand{\tp}{\otimes}
\newcommand{\ds}{\oplus}
\newcommand{\no}{\nonumber}
\newcommand{\DD}{\mathfrak{D}}
\newcommand{\LL}{\mathfrak{L}}
\newcommand{\FF}{\mathfrak{F}}
\newcommand{\dd}{\mathfrak{d}}
\newcommand{\ff}{\mathfrak{f}}
\newcommand{\cc}{\mathfrak{c}}
\newcommand{\nn}{\mathfrak{n}}
\newcommand{\mm}{\mathfrak{m}}
\newcommand{\coder}{\mathfrak{coder}}
\newcommand{\cohom}{\mathfrak{cohom}}
\newcommand{\morph}{\mathfrak{morph}}
\newcommand{\liftd}[1]{D(#1)}
\newcommand{\lifte}[1]{e^{#1}}
\newcommand{\vect}[1]{\accentset{\rightharpoonup}{#1}}
\newcommand{\T}{\mathsf{T}}

\title[Homotopy Algebras and Strings ]{On Homotopy Algebras and  Quantum String Field Theory} 

\author{Korbinian M\"unster}
\address{LMU\\ Arnold Sommerfeld Center for Theoretical Physics\\ Theresienstrasse 37\\ D-80333 Munich\\ Germany}

\email{korbinian.muenster(at)physik.uni-muenchen.de}


\author{Ivo Sachs}

\address{LMU\\ Arnold Sommerfeld Center for Theoretical Physics\\ Theresienstrasse 37\\ D-80333 Munich\\ Germany}


\email{Ivo.Sachs(at)lmu.de}

\begin{abstract}
We revisit the existence, background independence and uniqueness of closed, open and open-closed bosonic- and topological string field theory, using the machinery of homotopy algebra. 
In a theory of classical open- and closed strings, the space of inequivalent open string field theories is isomorphic to the space of classical closed string backgrounds. We then discuss obstructions of
these moduli spaces at the quantum level. For the quantum theory of closed strings, uniqueness on a given background follows from the decomposition theorem for loop homotopy algebras. 
We also address the question of  background independence of closed string field theory. 
\end{abstract}




\maketitle

\section{Introduction}
The standard formulation of classical string theory consists of a set of rules to compute scattering amplitudes for a set of $n$ (excited) strings typically propagating on a D-dimensional Minkowski space-time $M_D$. This prescription involves an integration over the moduli space of disks with $n$ punctures for open strings (or spheres with $n$ punctures for the closed strings). Comparing this with the approach taken for point particles the situation in string theory seems incomplete. Indeed, for point particles one starts with an action principle 
and then obtains the classical scattering amplitudes by solving the equations of motions deriving from this action. Since the various string excitations ought to be interpreted as particles one would hope to be able to apply the same procedure for the scattering of strings. 
The aim of string field theory is precisely to provide such an action principle so that the set of rules to compute scattering amplitudes for strings follow from this action. Since the string consists of a infinite linear superpositions of point particle excitations one would expect that such an action may be rather complicated. Yet the first construction of a consistent classical string field theory of interacting open strings \cite{Witten:1985cc} has a remarkably simple algebraic structure of a differential graded algebra (DGA) together with a non-degenerate odd symplectic form. 

The geometric approach for the construction of string field theory \cite{Zwiebach closed, Zwiebach open-closed}, starts with a decomposition of the relevant moduli space of Riemann surfaces into elementary vertices  and graphs. The condition that the moduli space is covered exactly once, implies that the geometric vertices satisfy a classical Batalin-Vilkovisky master equation. From this one then anticipates that any string field theory action should realize some homotopy algebra. The subject of this talk is to investigate to what extend this algebraic structure is useful, and to determine certain additional properties that should be satisfied by any consistent string field theory. In particular, it is of interest to know in what sense string field theory is unique. Another related issue stems from the fact that the construction of string field theory assumes that the string propagates in a certain  string background, whose geometry is that of Minkowski space. However, since string theory includes gravity, this background is dynamical. The question of background independence of this construction is thus relevant.

To set the stage, let us start with the well understood case of a single point particle  propagating on a non-compact manifold $M_D$ with a pseudo-Riemannian metric $g$. The world line of the particle is described by a curve $\phi: [a,b]\to M_D$ that extremizes the action 
\begin{align} 
S[\phi,h]=\int\limits_{[a,b]}\frac{1}{\sqrt{h_{tt}}}  g(\dot\phi,\dot\phi) dt\nonumber
  \end{align}
where $h_{tt}$ is a non-dynamical "metric" on the world line that can be set to 1 by a suitable reparametrization of $t$.  Similarly, for an open string we have a map $\phi: \Sigma=[a,b]\times [c,d]\to M_D$ that extremizes the action
\begin{align} \label{as}
S[\phi,h]=\int\limits_{\Sigma}\sqrt{h}h^{ij}g(\partial_i\phi,\partial_j\phi)
  \end{align}
  so that the area is minimal. If the Riemann curvature of $M_D$ vanishes, then the action (\ref{as}) is invariant under conformal mappings of the {\it world sheet} $\Sigma$. For $[a,b]=[-\infty,\infty]$, we can conformally map $\Sigma$ to a disk with $2$ punctures on its boundary. Analogously, a world sheet describing $n-1$ strings joining into one can be mapped into a disk with $n$ punctures. In order to specify which particles (or string excitations) are involved in the scattering amplitude we need to endow the puncture with additional structures. This is done by attaching conformal tensors $\{V_i[\phi]\}$ built out of the maps $\phi$  evaluated at the puncture and the coefficients of the Laurent polynomial of $\phi$ evaluated in local coordinates. The amplitude is then expressed in terms of  the $n$-point correlator
  \begin{align} \label{ccft}
<V_{i_1}(z_1),\cdots , V_{i_n}(z_n)> \;\text{,}
  \end{align}
with respect to the (formal) Gau\ss ian measure defined by $S[\phi]$. In fact the correlator (\ref{ccft}) which is called a conformal field theory correlator in physics is not quite what one needs. In order to get the string scattering amplitude we need to integrate over the moduli space of the punctured disk. Now, since the action $S[\phi,h]$ is invariant under diffeomorphisms on the world sheet $\Sigma$ as well as under Weyl re-scalings of the world sheet metric $h$ we really want to integrate over the $(n-3)$-dimensional gauge-fixed moduli space $M_{n-3}$ (for a review see e.g. \cite{Witten:2012bh} and references therein). Treating the gauge-fixed action using the standard BRST formalism we end up with an action $S[\phi,c,b]$ including odd world sheet tensors fields (BRST ghosts) together with an odd differential $Q_o$ that generates the odd symmetry transformations of the gauge fixed action. Similarly, the insertions at the punctures of $\Sigma$ contain added Laurent coefficients of the $b$ and $c$ ghosts. The string amplitude can be written schematically as in figure 1, 
\begin{figure}[h] \centering
\begin{minipage}[b]{0.4\textwidth}
\resizebox{4.5cm}{!}{
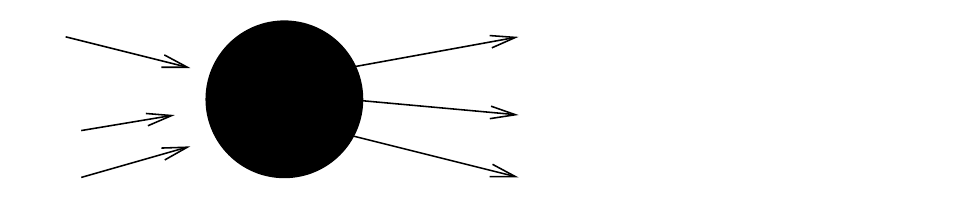}
\vspace{4mm}
\end{minipage}
\begin{minipage}[b]{0.4\textwidth}
\resizebox{3.5cm}{!}{
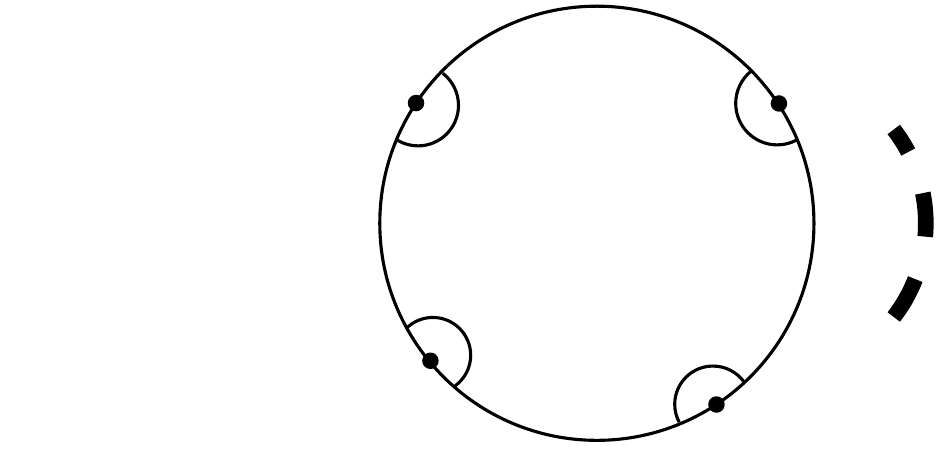}
\end{minipage}
\caption{Sketch of the CFT realization of the scattering amplitude of $n$ open strings.}
\end{figure}
where the $n-3$ meromorphic vector fields $v_i$  are constant near the puncture $P_i$, and
cannot be extended to the whole disk. These
vector fields generate translations in the moduli space; they move the punctures.  Concretely, this amplitude becomes
\begin{align}
\int_{M_{n-3}} ds_1 \ldots ds_{n-3} \, \langle  b(v_1) 
\ldots b(v_{n-3}) \, V_{i_1}[\phi,b,c](z_1) \ldots V_{i_n}[\phi,b,c](z_n)\rangle,
\label{Cdef}
\end{align}
where the correlator is evaluated with respect to the measure obtained from the world sheet action $S[\phi,c,b]$.
What we have just described is  what is usually referred to as  the operator formalism of the world sheet conformal field theory (CFT), which dresses the geometric amplitudes (punctured disks) with the physical states (particles). The amplitudes  (\ref{Cdef}) are well defined on the cohomology of $Q_o$. 

The purpose of string field theory is two-fold. First to reproduce these amplitudes in terms of vertices and graphs built from them and second to generalize the amplitudes  (\ref{Cdef})  on $coh(Q_o)$ to the module $A_o$ of  all conformal tensors with suitable regularity conditions. At the geometrical level, the simplest possible construction would be that of a single vertex of $3$ joining strings which has no moduli, with all amplitudes recovered from graphs built from $3$-vertices. This is indeed possible for the open bosonic string \cite{Witten:1985cc}. However, the decomposition of moduli space is not unique so that other realizations are possible where higher order vertices are needed to recover the amplitudes (\ref{Cdef}). In any case the geometric vertices  in any consistent decomposition form a BV algebra.  

The world sheet CFT  then defines a morphism of BV algebras between the set of geometric vertices $\{\mathcal{V}_n\}$, and the dressed "physical" vertices. It also provides us with an inner product on the graded module $A_o$ generated by the conformal tensors $V_{i}[\phi,b,c]$ of the $(\phi,b,c)$ - CFT inserted at the origin in the local coordinate $z$ around a puncture $P$ on the disk. With the help of the latter we can interpret the set of physical vertices as multilinear maps $m_i: A_o^{\otimes i} \to A_o$, $m_1=Q_o$, with some further symmetry properties implied by the cyclic symmetry of the vertices. We denote by  $C(A_o)$, the space of such multilinear maps on $A_o$. 
 It is then not hard to see that the BV-master equation implies that the maps $m_i$ define an $A_\infty$-structure. One way to see this is to define a coderivation $M$ of degree 1 on the tensor algebra $TA_o=\oplus_n A_o^{\otimes n}$ with components
\begin{align}
  (M)_{n,u}=\sum\limits_{\stackrel{r+s+t = n}{r+1+t = u}} 
  {\bf 1}^{\otimes r}\otimes m_s\otimes  {\bf 1}^{\otimes t}\nonumber
\end{align}
Imposing vanishing of  the graded commutator $[M,M]$, we obtain a characterization of all 
differentials compatible with the $A_\infty$-structure. 

The 
classical solutions of the string field theory action defined by the maps $m_i$ together with $<\cdot ,\cdot>$ are given by the Maurer-Cartan elements, $M(e^{\psi_0})=0$. 

There is an analogous story for classical closed strings obtained from the above by replacing the punctured disk by a punctured sphere with world sheet conformal field theory $S[\Phi, c ,\bar{c} ,b,\bar{b}]$ and dressed by conformal tensors $V_{i}[\Phi,b,\bar b,c,\bar c]$ where $b,\bar b,c$ and $\bar c$ depend holomorphically and anti-holomorphically on the world sheet coordinates $z$ and $\bar z$, respectively. The CFT then provides a morphism between the set of geometric vertices and the (dressed) physical  vertices of closed string field theory. The latter can again be interpreted as maps, $l_i$ on the garaded symmetric module $SA_c=\oplus_n A_c^{\wedge n}$. They accordingly realize an $L_\infty$ algebra $(A_c,L)$, with $[L,L]=0$.

Finally, we let open and closed strings interact with each other. The open closed vertices consist of disks with punctures on the boundary as well as on the disk. These vertices realize an $L_\infty$ morphism $F$, between the closed and open sector taken separately, 
\begin{align}\label{eq:ocha}
(A_c,L)\xrightarrow[]{\text{$F$}} (\op{Coder}^{cycl}(TA_o),d_h,[\cdot,\cdot])\;\text{}\,.
\end{align}
This is the open-closed homotopy algebra of Kajiura and Stasheff \cite{Kajiura:2004xu}.

\begin{remark}\label{rem:1}
Note that, while the geometric decomposition of the moduli spaces appearing in the construction of string field theory, just reviewed, is independent of the details of $M_D$ the operator formalism makes explicit use of the geometry of $M_D$ as well as possible other background fields inserted at the punctures. In particular, the module $A$ of conformal tensors typically depends on these data. This is where the background dependence enters in the construction of string field theory. This is in contrast to e.g. General Relativity where the action does not depend on any background metric on $M_D$.   
\end{remark}

A natural question that arises in the above context is whether for  a given background (in the sense just described) the generalization of  (\ref{Cdef}) as well as its closed string version is unique. For classical string field theory the answer to this question is affirmative, as follows form  the decomposition theorem \cite{Kajiura:2004xu} for homotopy algebras. This theorem establishes an isomorphism between a given homotopy algebra and the direct sum of a linear contractible algebra and a minimal model. In the context of string field theory, the structure maps of the minimal model are given by (\ref{Cdef}). 

In this talk we discuss the following generalizations of the results reviewed above:
\begin{itemize}
\item classification of inequivalent deformations of classical open string field theory. 
\item background independence of closed string field theory.
\item decomposition theorem for quantum closed string field theory.
\item quantization of the open closed homotopy algebra.
 \end{itemize}

%
%
%

\section{Results}
Let us start with non-trivial deformations of open string field theory. That is we consider  {\it continuous} deformations of the worldsheet CFT that do not preserve $Q_o$ and   (\ref{Cdef}) simultaneously . The usefulness of the homotopy formulation of SFT in this respect is that this problem can be formulated as a cohomology problem. Indeed, since any consistent open string field theory realizes an $A_\infty$ algebra, i.e. defines a coderivation $M$ of degree 1 on the tensor algebra $TA_o$ with $[M,M]=0$, any infinitesimal deformation $M+\delta M$ satisfies $d_H(\delta M)\equiv [M,\delta M]=0$. For a given  worldsheet CFT one would therefore like to determine $coh(d_H)$. The outcome of this analysis is contained in 
\begin{theorem}[ \cite{Moeller:2010mh} ]\label{thm:01}
Let  $S[\phi,c,b]$ be the open string world sheet CFT on $M_D$,  $A_o$ the corresponding module of conformal tensors,  $Q_o$ the BRST differential, and  (\ref{Cdef}) the corresponding string amplitudes on $coh(Q_o)$. Then the only non-trivial infinitesimal deformations of $S[\phi,c,b]$ preserving $A_o$ are infinitesimal deformations of the closed string background in the relative cohomology of $Q_c$,
\begin{align}\label{dhqc}
coh(d_H)\cong coh(b_0-\bar b_0,Q_c) \;\text{.}\nonumber
\end{align}
\end{theorem}

\begin{remark}\label{rem:2}
A particular class of deformations that do not preserve  $Q_o$ and  (\ref{Cdef}) are shifts in the open string background $\phi_0\to \phi_0+\epsilon\delta \phi$ with $M(e^{ \phi_0+\epsilon\delta \phi})=O(\epsilon^2)$. Such transformations are, however, $d_H$-exact as are all field redefinitions of $\phi$. From a physics perspective, the interesting fact implied by theorem \ref{thm:01} is that open string theory already contains the complete information of the particle content of {\it closed} string theory. 
\end{remark}

\begin{proof} The proof of this assertion proceeds via a detailed analysis of the deformations of the CFT correlator  (\ref{Cdef}). 
\end{proof}

Given the isomorphism between the cohomologies one may wonder whether this isomorphism holds for finite deformations. On the closed string side finite deformations correspond to classical solutions of the closed string field theory equation of motion, that is Maurer-Cartan elements $L(e^{ \Phi})=0$, whereas finite deformations of open string field theory are Maurer-Cartan elements of $[\cdot,\cdot]$ on $\{M\in \op{Coder}^{cycl}(TA_o)\}$, that is $[M,M]=0$. A classic theorem of Kontsevich then guarantees isomorphism at the finite level, or more precisely that the moduli spaces of two $L_\infty$-algebras connected by a $L_\infty$-quasi-isomorphism are isomorphic. Thus, we have 
\begin{corollary}\label{col:1}
 Let $\mathcal{M}(A_c,L)$ and $\mathcal{M}(\op{Coder}^{cycl}(TA_o),[\cdot,\cdot])$ be the moduli space of Maurer-Cartan elements obtained by moding out $L$- and $[\cdot,\cdot]$ -gauge transformations respectively, then we have 
 \be\label{mm}
 \mathcal{M}(A_c,L)\cong \mathcal{M}(\op{Coder}^{cycl}(TA_o),d_h,[\cdot,\cdot]) \;\text{.}\nonumber
 \ee
\end{corollary}

 We will return to the question whether this isomorphism survives quantization below but first we would like to turn to background independence of closed string field theory. As mentioned above for a given background the operator formalisms realizes a certain $L_\infty$ algebra. Furthermore, for a given classical solution $\Phi_0$ in this field theory we then obtain a new homotopy algebra upon conjugation by this Maurer-Cartan element. Background independence then would imply that the structure maps of the minimal model obtained from this homotopy algebra are equivalent to the amplitudes   (\ref{Cdef}) obtained with the measure of the world-sheet CFT $S[\Phi,c,\bar c,b,\bar b]$ in the new background (see figure 2). 
 
 \begin{figure}
\begin{tikzpicture}[node distance=5cm, auto]
  \node (S1) {$S_{CFT}$};
  \node (S2) [below of=S1] {$S^\prime_{CFT}$};
  \node (L1) [right of=S1] {$\{l_n\}_{n\in \N}$};
  \node (L2) [right of=S2] {$\{l_n^\prime\}_{n \in \N}$};
  \draw[->] (S1) to node [swap] {$S_{CFT}\to S_{CFT}+\int \Phi_0$} (S2);
  \draw[->] (S1) to node {operator formalism} (L1);
  \draw[->] (S2) to node {operator formalism} (L2);
  \draw[->] (L1) to node {$L\to e^{-\Phi_0}\circ L \circ e^{\Phi_0}$} (L2);
  \path (S1) to node [swap] {$\mathcal{V}_n\mapsto l_n$} (L1);
  \path (S2) to node [swap] {$\mathcal{V}_n\mapsto l^\prime_n$} (L2);
  \end{tikzpicture}
\caption{Background independence requires that the $L_\infty$ maps $\{l_n'\}$ obtained upon  conjugation by the MC-element $e^{\Phi_0}$ are equivalent to those obtained from the world sheet CFT in the background $\Phi_0$.}
\label{fig:back}
\end{figure}

 We can answer this question by addressing the cohomology problem on $\{L\in \op{Coder}^{sym}(SA_c)\}$. The bracket $[\cdot , \cdot ]$ on $\op{Coder}(SA)$ induces the Chevalley-Eilenberg differential $d_C = [L,\cdot ]$ on the deformation complex. The analysis proceeds in close analogy with that for open string theory with the result, 
 \begin{proposition}\label{prop:1}
Let  $S[\Phi,c,\bar c,b,\bar b]$ be the closed string world sheet CFT on $M_D$,  $A_c$ the corresponding module of conformal tensors and  $Q_c$ the BRST differential. Then 
\begin{align}\label{cec}
coh(d_c)=\emptyset \;\text{.}\nonumber
\end{align}
\end{proposition}
 
 An immediate consequence of this proposition is that the diagram in figure 2 commutes which, in turn, implies independence under shifts in the  background that preserve $A_c$. 

\begin{remark}\label{rem:3}
We should note that generic shifts in the background $\Phi$ will not preserve the module $A_c$. 
\end{remark}

Let us now return to the decomposition theorem which states that a homotopy algebra defined on a certain complex can be decomposed into the direct sum of a minimal and a linear contractible part. By definition, the linear contractible part is just a complex with vanishing cohomology, whereas the minimal part is a homotopy algebra of the same type as the initial one but without differential \cite{Kontsevich}. Furthermore, the initial and the decomposed algebra are isomorphic in the appropriate sense. Clearly, the minimal part can be extracted from the decomposed algebra by projection, and thus the decomposition theorem implies the minimal model theorem. 

The relevance of the minimal model theorem in physics is as follows: Suppose that the vertices of some field theory satisfy the axioms of some homotopy algebra. Then the minimal model describes the corresponding S-matrix amplitudes \cite{Kajiura open1, Kajiura open2}. Furthermore, the S-matrix amplitudes and the field theory vertices are quasi-isomorphic, which implies that their respective moduli spaces are isomorphic (this follows in general from the minimal model theorem). 

Now we conclude that string field theory is unique up to isomorphisms on a fixed conformal background (CFT): In string field theory, the differential is generically given by the BRST charge $Q$. Furthermore the CFT determines the S-matrix amplitudes. Thus a conformal background determines the minimal and the linear contractible part, which implies uniqueness up to isomorphisms.

An explicit construction of the decomposition model is known for the classical algebras ($A_\infty$ and $L_\infty$) \cite{Kajiura open1, Kajiura open2}. In the following we construct the decomposition model for quantum closed string field theory, formulated in the framework of $IBL_\infty$-algebras (see e.g. \cite{Cieliebak ibl,Munster:2011ij} for a definition).

Quantum closed string field theory has the algebraic structure of a loop homotopy Lie-algebra $(A,\LL)$  \cite{Markl loop}, i.e. 
\be\label{eq:loopII}
\LL=\sum \hbar^g L^g +\hbar \Omega^{-1} \;\text{,}\quad \LL^2=0 \;\text{,} 
\ee
where $L^g=\liftd{l^g}\in \op{Coder}^{cycl}(SA)$ and $\Omega^{-1}=\liftd{\omega^{-1}}\in \op{Coder}^2(SA)$ is the lift of the inverse of the odd symplectic structure ($D$ denotes the lift from multilinear maps to coderivations). 
We define $l_q\defineL\sum_g \hbar^g l^g$. The differential on $A$ is given by $d=l_{cl}\circ i_1$. Furthermore we abbreviate the collection of multilinear maps without the differential by $l_q^\ast\defineL l_q-d$. 

\begin{definition}
A pre Hodge decomposition of $A$ is a map $h:A\to A$ of degree minus one which is compatible with the symplectic structure and squares to zero. 
\end{definition}
For a given pre Hodge decomposition of $A$, we define the map 
\be\label{eq:hodge}
P=1+dh+hd \;\text{,}
\ee
and
\be\label{eq:metric}
g\defineL -\omega\circ d \mand g^{-1}\defineL h\circ\omega^{-1} \in A^{\w 2} \;\text{,}
\ee
where the symplectic structure $\omega$ and its inverse $\omega^{-1}$ are considered as a map from $A$ to $A^\ast$ and $A^\ast$ to $A$, respectively.
We define \emph{trees} constructed recursively from $l^\ast_q$ and $h$ via
\be\label{eq:treesq}
\mathsf{T}_{q}=h\circ l^\ast_{q}\circ e^{1+\mathsf{T}_{q}} \mand \mathsf{T}_{q}\circ i_1 =0 \;\text{.}
\ee

\begin{theorem}[\cite{Munster:2012gy}]\label{thm:dec}
Let $(A,\LL=D(d+l^\ast_q +\hbar \omega^{-1}))$ be a loop homotopy Lie algebra. For a given pre Hodge decomposition $h$, there is an associated loop homotopy Lie algebra 
\be\label{eq:deccoder}
\bar{\LL}=D(d+ \overset{(P)}{\mathsf{T}}_q\circ e^{\hbar g^{-1}} +\hbar \bar{\omega}^{-1}) \;\text{,}
\ee 
where $\bar{\omega}^{-1}=P^{\w2}(\omega^{-1})$ and $\overset{(P)}{\mathsf{T}}_q\circ e^{\hbar g^{-1}} $ represents the graphs with a single output labeled by $P$. Furthermore there is an $IBL_\infty$-isomorphism from $(A,\bar{\LL})$ to $(A,\LL)$. $d$ is called the linear contractible part and $\overset{(P)}{\mathsf{T}}_q\circ E(\hbar g^{-1}) +\hbar \bar{\omega}^{-1}$ the minimal part.
\end{theorem}
\begin{proof}
  The proof follows by explicit verification, using equation (\ref{eq:loopII}), (\ref{eq:hodge}) and (\ref{eq:treesq}).
\end{proof}

Finally, we describe the quantum generalization of the classical open-closed homotopy algebra (OCHA) of Kajiura and Stasheff. As already alluded in the introduction, the OCHA can be described by an $L_\infty$-morphism, $N$, mapping from the closed string algebra $(A_c,L)$ to the deformation complex of the open string algebra $(\op{Coder}^{cycl}(TA_o),d_h,[\cdot,\cdot])$, i.e.
\be\no
e^N\circ L= D(d_h+[\cdot,\cdot])\circ e^N \;\text{,}
\ee
or equivalently
\be\label{eq:ochaex}
N\circ L= d_h\circ N +\inv{2} [N,N]\circ \Delta \;\text{,}
\ee
where $N$ describes the open-closed vertices and the comultiplication $\Delta:TA\to TA\tp TA$ is defined by
\be\no
\Delta (a_1\tp\dots\tp a_n)= \sum_{i=0}^n (a_1\tp \dots\tp a_i)\tp(a_{i+1}\tp\dots\tp a_n)\;\text{.}
\ee

In a similar way one can describe the QOCHA by an $IBL_\infty$-morphism from the loop homotopy Lie algebra $(A_c,\LL)$ of closed strings to the involutive Lie bialgebra $(\mathcal{A}_o,d_h,[\cdot,\cdot],\delta)$, where $\mathcal{A}_o\defineL \op{Hom}^{cycl}(TA_o,\Bbbk)$ \footnote{In the quantum case it is more convenient to work with $\op{Hom}^{cycl}(TA_o,\Bbbk)$ rather than with $\op{Coder}^{cycl}(TA_o)$}. 

The operation
\be\no
\delta:\mathcal{A}_o \to \mathcal{A}_o^{\w 2} \;\text{,}
\ee
is defined by 
\begin{align}\label{eq:delta}
(\delta f)&(a_1,\dots,a_n)(b_1,\dots,b_m)\\\no
\defineL&(-1)^f\sum_{i=1}^{n}\sum_{j=1}^{m}(-1)^\epsilon f(e_k,a_{i},\dots,a_n,a_1,\dots,a_{i-1},e^k,b_{j},\dots,b_m,b_1,\dots,b_{j-1})\;\text{,}
\end{align}
where $(-1)^\epsilon$ denotes the Koszul sign, $\{e_i\}$ is a basis of $A_o$ and $\{e^i\}$ is the corresponding dual basis satisfying $\omega_o({}_ie,e^j)={}_i\delta^j$.
This operation can be interpreted geometrically as the sewing of open strings on one boundary component.
In \cite{Chen liebi,Cieliebak ibl} it has been shown that $(\mathcal{A}_o,d_h,[\cdot,\cdot],\delta)$ defines an involutive Lie bialgebra, a special case of an $IBL_\infty$-algebra. In the language of $IBL_\infty$-algebras this is equivalent to the statement that
\be\no
\LL_o\defineL \liftd{d_h+[\cdot,\cdot]+\hbar\, \delta} 
\ee
squares to zero.

\begin{definition}[\cite{Munster:2011ij}] The quantum open-closed homotopy algebra is defined by an $IBL_\infty$-morphism from a loop homotopy Lie algebra $(A_c,\LL_c)$ to the involutive Lie bialgebra
$(\mathcal{A}_o,\LL_o)$, i.e.
\be\label{eq:qocha1}
\lifte{\frak{n}}\circ \LL_c=\LL_o\circ\lifte{\frak{n}}
\ee
\end{definition}
The maps $\frak{n}$ describe the open-closed vertices to all orders in $\hbar$.

Equation (\ref{eq:qocha1}) can be recast, such that the five distinct sewing operations in open-closed string field theory become apparent:
\begin{align}\label{eq:qocha2}
\frak{n}&\circ \LL_c + \frac{\hbar}{2}\bracketi{\frak{n}\circ \liftd{e_i}\w \frak{n}\circ \liftd{e^i}}\circ\Delta \\\no
&= \LL_o \circ \frak{n} +\inv{2}D({[\cdot,\cdot]})\circ \bracketi{\frak{n}\w \frak{n}}\circ\Delta -\bracketi{(D({[\cdot,\cdot]})\circ\frak{n})\w\frak{n}}\circ\Delta \;\text{.}
\end{align}
In equation (\ref{eq:qocha2}), $e_i$ and $e^i$ denote a basis and corresponding dual basis of $A_c$ w.r.t. the symplectic structure $\omega_c$. Obviously we recover the OCHA of equation 
(\ref{eq:ochaex}) in the limit $\hbar\to 0$.

Similarly as in the classical case, the morphism $\lifte{\frak{n}}$ is a quasi-isomorphism which implies isomorphism of the corresponding moduli spaces, i.e.
\be\no
\mathcal{M}(A_c,{\LL}_c)\cong \mathcal{M}(A_o,{\LL}_o)\;\text{.}
\ee

\begin{theorem}[\cite{Munster:2011ij}]\label{qmoduli}
The moduli space of any loop homotopy Lie algebra is empty, 
\be\no
\mathcal{M}(A_c,{\LL}_c) = \emptyset \;\text{.}
\ee
\end{theorem}
\begin{proof}
The proof follows from considering the order $\hbar$ term of the Maurer Cartan equation for a general ansatz. This equation, together with the non-degeneracy of the symplectic form implies triviality of the cohomology, which in turn implies that $\mathcal{M}(A_c,{\LL}_c) = \emptyset$.
\end{proof}

\begin{remark}\label{}
The story is different for the topological string, where the symplectic structure $\omega$ degenerates on-shell.  Under this condition, theorem \ref{qmoduli} does not hold anymore, which implies consistency of open topological string theory at the quantum level in contrast to bosonic string theory.
\end{remark}


\section*{Acknowledgement} 
The authors would like to thank Barton Zwiebach for many informative discussions and subtle remarks as well as Branislav Jurco, Kai Cieliebak and Sebastian Konopka for helpful discussions. K.M. would like to thank  Martin Markl and Martin Doubek who stimulated his interest in operads and their applications to string field theory. This project was supported in parts by the DFG Transregional Collaborative Research Centre TRR 33,
the DFG cluster of excellence ``Origin and Structure of the Universe'' as well as the DAAD project  54446342.

\end{document}

%% file: amp.pdf_t
\begin{picture}(0,0)%
\includegraphics{amp.pdf}%
\end{picture}%
\setlength{\unitlength}{3947sp}%
\begingroup\makeatletter\ifx\SetFigFont\undefined%
\gdef\SetFigFont#1#2#3#4#5{%
  \reset@font\fontsize{#1}{#2pt}%
  \fontfamily{#3}\fontseries{#4}\fontshape{#5}%
  \selectfont}%
\fi\endgroup%
\begin{picture}(4653,1005)(1711,-301)
\put(1726,539){\makebox(0,0)[lb]{\smash{{\SetFigFont{12}{14.4}{\rmdefault}{\bfdefault}{\updefault}{\color[rgb]{0,0,0}1}%
}}}}
\put(1876, 89){\makebox(0,0)[lb]{\smash{{\SetFigFont{12}{14.4}{\rmdefault}{\bfdefault}{\updefault}{\color[rgb]{0,0,0}2}%
}}}}
\put(1801,-286){\makebox(0,0)[lb]{\smash{{\SetFigFont{12}{14.4}{\rmdefault}{\bfdefault}{\updefault}{\color[rgb]{0,0,0}3}%
}}}}
\put(4426,-211){\makebox(0,0)[lb]{\smash{{\SetFigFont{12}{14.4}{\rmdefault}{\bfdefault}{\updefault}{\color[rgb]{0,0,0}n-2}%
}}}}
\put(4351, 89){\makebox(0,0)[lb]{\smash{{\SetFigFont{12}{14.4}{\rmdefault}{\bfdefault}{\updefault}{\color[rgb]{0,0,0}n-1}%
}}}}
\put(4351,464){\makebox(0,0)[lb]{\smash{{\SetFigFont{12}{14.4}{\rmdefault}{\bfdefault}{\updefault}{\color[rgb]{0,0,0}n}%
}}}}
\put(6226,164){\makebox(0,0)[lb]{\smash{{\SetFigFont{12}{14.4}{\rmdefault}{\bfdefault}{\updefault}{\color[rgb]{0,0,0}=}%
}}}}
\end{picture}%

%% file: openp.pdf_t
\begin{picture}(0,0)%
\includegraphics{openp.pdf}%
\end{picture}%
\setlength{\unitlength}{3947sp}%
\begingroup\makeatletter\ifx\SetFigFont\undefined%
\gdef\SetFigFont#1#2#3#4#5{%
  \reset@font\fontsize{#1}{#2pt}%
  \fontfamily{#3}\fontseries{#4}\fontshape{#5}%
  \selectfont}%
\fi\endgroup%
\begin{picture}(4489,2167)(286,-4622)
\put(4051,-2836){\makebox(0,0)[lb]{\smash{{\SetFigFont{20}{24.0}{\familydefault}{\mddefault}{\updefault}{\color[rgb]{0,0,0}$V_{i_n}$}%
}}}}
\put(3976,-4486){\makebox(0,0)[lb]{\smash{{\SetFigFont{20}{24.0}{\familydefault}{\mddefault}{\updefault}{\color[rgb]{0,0,0}$V_{i_3}$}%
}}}}
\put(1801,-2836){\makebox(0,0)[lb]{\smash{{\SetFigFont{20}{24.0}{\familydefault}{\mddefault}{\updefault}{\color[rgb]{0,0,0}$V_{i_1}$}%
}}}}
\put(2551,-3061){\makebox(0,0)[lb]{\smash{{\SetFigFont{20}{24.0}{\familydefault}{\mddefault}{\updefault}{\color[rgb]{0,0,0}$b(v_1)$}%
}}}}
\put(2551,-4261){\makebox(0,0)[lb]{\smash{{\SetFigFont{20}{24.0}{\familydefault}{\mddefault}{\updefault}{\color[rgb]{0,0,0}$b(v_2)$}%
}}}}
\put(3376,-4111){\makebox(0,0)[lb]{\smash{{\SetFigFont{20}{24.0}{\familydefault}{\mddefault}{\updefault}{\color[rgb]{0,0,0}$b(v_3)$}%
}}}}
\put(1801,-4486){\makebox(0,0)[lb]{\smash{{\SetFigFont{20}{24.0}{\familydefault}{\mddefault}{\updefault}{\color[rgb]{0,0,0}$V_{i_2}$}%
}}}}
\put(451,-3511){\makebox(0,0)[lb]{\smash{{\SetFigFont{48}{57.6}{\familydefault}{\mddefault}{\updefault}{\color[rgb]{0,0,0}$\int\limits_{M_{n-3}}$}%
}}}}
\end{picture}%